\documentclass[smallextended]{svjour3m}       

\usepackage{tikz}
\usepackage{pgfplots}
\usepackage{verbatim}

\newcommand{\grad}{\mathop{\rm grad}\nolimits}
\renewcommand{\div}{\mathop{\rm div}\nolimits}
\newcommand{\const}{\mathop{\rm const}\nolimits}

\smartqed  
\usepackage{graphicx}

\journalname{arXiv.Org}
\begin{document}

\title{Domain decomposition schemes for evolutionary equations of first order with not self-adjoint operators}

\titlerunning{Domain decomposition schemes with not self-adjoint operators}        

\author{Petr N. Vabishchevich}

\institute{P.N. Vabishchevich \at
              Keldysh Institute of Applied Mathematics, 
              4 Miusskaya Square, 125047 Moscow, Russia \\
              Tel.: +7-499-9781314\\
              Fax: +7-499-9720737\\
              \email{vabishchevich@gmail.com} 
}

\date{Submitted to arXiv.org January 12, 2011}

\maketitle

\begin{abstract}
Domain decomposition methods are essential in solving applied problems on parallel computer systems.  
For boundary value problems for evolutionary equations the implicit schemes are in common use
to solve problems at a new time level employing  iterative methods of domain decomposition.  
An alternative approach is based on constructing iteration-free methods based 
on special schemes of splitting into subdomains.  
Such regionally-additive schemes are constructed using the general theory of additive operator-difference schemes.  
There are employed the analogues of classical schemes of alternating direction method, 
locally one-dimensional schemes, factorization methods, vector and regularized additive schemes.   
The main results were obtained here for time-dependent problems with self-adjoint elliptic operators of second order.  

The paper discusses the Cauchy problem for the first order evolutionary equations with a 
nonnegative not self-adjoint operator in a finite-dimensional Hilbert space.  
Based on the partition of unit, we have constructed the operators of decomposition which preserve nonnegativity
for the individual operator terms of splitting.  
Unconditionally stable additive schemes of domain decomposition were constructed
using the regularization principle for operator-difference schemes.
Vector additive schemes were considered, too.  
The results of our work are illustrated by a model problem for the two-dimensional parabolic equation.  

\keywords{Time-dependent problems \and Domain decomposition method 
\and Additive schemes \and Operator-splitting difference schemes}
\PACS{02.60.Lj \and 02.70.Bf}
\subclass{65N06 \and 65M06}
\end{abstract}

\section{Introduction}
\label{s-1}

Domain decomposition methods are widely used for the numerical solution of boundary value problems 
for partial differential equations on parallel computers.  
Stationary problems  \cite{pre05281749,0931.65118,0857.65126,1069.65138} are the most extensively studied 
in the theory of domain decomposition methods.
Numerical algorithms with and without overlapping of subdomains are used here
in the synchronous (sequential) and asynchronous (parallel) organization of computations.  

Domain decomposition methods for unsteady problems are based on two approaches  \cite{1018.65103}.  
In the first approach for the approximate solution of time-dependent problems we use the
standard implicit approximations in time.    
After that domain decomposition methods are applied to solve the discrete problem at a new time level.  
For the optimal iterative methods of domain decomposition the number of iterations does not depend on the 
discretization steps in time and space \cite{Cai:1991:ASA,Cai:1994:MSM}.
The iteration-free domain decomposition algorithms are constructed for unsteady problems in the second approach.  
In some cases we can confine ourselves to only one iteration of the Schwarz alternating method 
for solving  boundary value problems for the parabolic equation of second order \cite{0825.65066,0766.65089}.  
Special schemes of splitting into subdomains  (regionally-additive schemes \cite{0719.65072,0723.65076})
are constructed, too.

Construction and convergence investigation of the regionally-additive schemes is
based on the general  results of the theory of splitting schemes  
\cite{Marchuk:1990:SAD,0971.65076,0209.47103}.
The most interesting for practice is the situation where the problem operator is decomposed into
a sum of three or more noncommutative not self-adjoint operators.  
In the case of such a multi-component splitting  the stable additive splitting schemes are constructed on the basis of the concept 
of summarized  approximation.  
Additively-averaged schemes of summarized  approximation are interesting for using on parallel computers.  
In the class of splitting schemes of full approximation \cite{0963.65091} we highlight the vector additive schemes 
based on the transition from the single original equation to a system of similar equations  
\cite{0712.65089,abrashin1998numerical,vabishchevich1996vector}.
Additive regularized operator-difference schemes are constructed in the most simple way for  multi-component splitting
\cite{samarskii1998regularized,samarskii1992regularized}, 
where stability is achieved via perturbations of operators of the difference scheme.  

Peculiarities of domain decomposition schemes result from  the selection of  splitting operators.  
To construct the operators of decomposition for boundary value problems for partial differential equations,
it is convenient to use the partition of unit for the computational domain  \cite{Dryja:1991:SMP,Laevsky,0863.65056,0719.65072,vab_138,0838.65086,0986.65510}.
In the  domain decomposition method with overlapping a separate subdomain is
associated with a function with values lying between zero and one.  
Domain decomposition methods for unsteady convection-diffusion problems are studied in works  \cite{0928.65102,vab_255,0888.65097}.
In the limiting case the width of subdomain overlapping is equal to the discretization step.  
In this case the regionally-additive schemes are interpreted as the decomposition 
without overlapping of subdomains but with appropriate boundary conditions of exchange.  
Domain decomposition methods for unsteady boundary value problems 
are summarized in the books  \cite{1018.65103,0963.65091}. 
More recent studies are presented in the work \cite{1156.65084}.
In this case we use different constructions for the splitting operators and for operators of 
the grid problem at a new time level.  

In this paper we construct domain decomposition schemes for the first order evolutionary 
equations with a general nonnegative operator in a finite Hilbert space.  
Decomposition operators are constructed here separately for the self-adjoint and skew-symmetric parts
using the partition of unit in the appropriate spaces.  Two classes of unconditionally stable regionally-additive 
regularized schemes are proposed. Vector additive operator-difference schemes of domain decomposition are considered.  
The paper is organized as follows.  
In section 2 there is formulated the Cauchy problem for the evolutionary equation of first order and 
the corresponding a priori estimate of stability is derived.  
Decomposition operators are constructed in Section 3.  
Problems with the self-adjoint operator and skew-symmetric one are considered separately.  
Unconditionally stable regularized additive  schemes of domain decomposition are constructed in Section 4, 
with the additive and multiplicative perturbation of the operator of transition to the new time level.
Vector splitting schemes are discussed in Section 5.  
In section 6 we consider a model boundary value problem for the two-dimensional parabolic equation 
along with the results of using different domain decomposition schemes.  
The main results are summarized in Section 7.  

\section{The Cauchy problem for the first order evolutionary equation}
\label{s-2}

Let $H$ be a finite-dimensional  real Hilbert space of grid functions with the scalar product and norm
 $(\cdot,\cdot)$ è $\|\cdot\|$, respectively.  
Let a constant (independent of time $t$)  grid operator  $A$ is nonnegative in $H$:
\begin{equation}\label{1}
  A \geq 0, 
  \quad \frac{d}{d t} A = A \frac{d}{d t} 
\end{equation}
and $E$  is the identity operator in  $H$.
We search the solution of the Cauchy   problem
\begin{equation}\label{2}
  \frac{d  u}{d  t} + A u = f(t),    
  \quad 0 < t \leq T ,
\end{equation}
\begin{equation}\label{3}
  u(0) = u^0.
\end{equation}

Problem (\ref{1})--(\ref{3})  results from a finite-difference approximation 
in space for the approximate solution of boundary value problems for partial differential equations.  
Similar systems of ordinary differential equations arise in using the finite-element method as well as 
in applying the finite-volume approach.  
Let us obtain the standard a priori estimate for problem (\ref{1})--(\ref{3}).  

Multiply equation (\ref{2}) by $u$ scalarly in  $H$ .  
In view of (\ref{1}) we obtain inequality  
 \begin{equation}\label{4}
  \frac{1}{2} \frac{d}{d t} \|u\|^2 \leq (f,u) .
\end{equation}
Taking into account  
\[
  (f, u) \leq \|f\| \|u\|,
\]
from (\ref{4})  we have  
\[
  \frac{d}{d t} \|u\| \leq \|f\| .
\]
In view of the Gronwall lemma we obtain the desired estimate  
\begin{equation}\label{5}
  \|u\| \leq \|u^0\| + \int_{0}^{t} \|f(\theta)\| d \theta ,
\end{equation}
which expresses the stability of the solution with respect to the initial data and right-hand side.  

The emphasis of our work is on constructing approximations in time for equation (\ref{2}).
Two-level schemes will be considered.  
Let $\tau$ be a step of the uniform grid in time and let   $y^n = y(t^n), \ t^n = n \tau$,
$n = 0,1, ..., N, \ N\tau = T$.
Equation (\ref{2}) is approximated by the two-level scheme with weights  
\begin{equation}\label{6}
  \frac{y^{n+1} - y^{n}}{\tau }
  + A(\sigma y^{n+1} + (1-\sigma) y^{n}) = \varphi^n,
  \quad n = 0,1, ..., N-1,
\end{equation}
where, for example,   $\varphi^n = f(\sigma t^{n+1} + (1-\sigma) t^{n})$.
It is supplemented by the initial condition  
\begin{equation}\label{7}
  y^0 = u^0 .
\end{equation}
Difference scheme (\ref{6}), (\ref{7}) has approximation error  $\mathcal{O} (\tau^2 + (\sigma - 0.5) \tau)$. 

Grid analog of (\ref{5})  is the estimate  at the time level 
\begin{equation}\label{8}
  \|y^{n+1}\| \leq \|y^{n}\| + \tau \|\varphi^n\|,
  \quad n = 0,1, ..., N-1 .
\end{equation}
Let us prove the following statement.  

\begin{theorem} 
\label{t-1} 
Difference scheme (\ref{1}), (\ref{6}), (\ref{7})
is unconditionally stable at $\sigma \geq 0.5$,
and estimate  (\ref{8})  holds  for the difference solution.  
\end{theorem} 

\begin{proof}
We write (\ref{6})  in the form  
\begin{equation}\label{9}
  y^{n+1} = S y^{n} + \tau (E + \sigma \tau A)^{-1} \varphi^n ,
\end{equation}
where  
\begin{equation}\label{10}
  S = (E + \sigma \tau A)^{-1} (E - (1 - \sigma) \tau A)
\end{equation}
is the operator of transition to the new time level.  
From (\ref{9}) we have  
\begin{equation}\label{11}
  \|y^{n+1}\| = \|S \| \, \|y^{n}\| + 
  \tau \|(E + \sigma \tau A)^{-1} \varphi^n\| .
\end{equation}

For the last term in the right-hand side of (\ref{11}) in the class of operators (\ref{1}), 
under natural conditions $\sigma \geq 0$  we have  
\[
  \|(E + \sigma \tau A)^{-1} \varphi^n\| \leq \|\varphi^n\|.
\]
We show that if $\sigma \geq 0.5$ then for the nonnegative operators $A$ the following estimate holds  
\begin{equation}\label{12}
  \|S\| \leq 1.
\end{equation}
In the Hilbert real space $H$   inequality (\ref{12})
 is equivalent  \cite{1152.15001} to fulfilment  of the operator inequality  
\[
  S S^* \leq E.
\]
In view of (\ref{10}) this inequality takes the form  
\[
  (E + \sigma \tau A)^{-1} (E - (1 - \sigma) \tau A)
  (E - (1 - \sigma) \tau A^*)  (E + \sigma \tau A^*)^{-1} \leq E. 
\]
Multiplying this inequality on the left by $(E + \sigma \tau A)^{-1}$  and on the right by $(E + \sigma \tau A^*)^{-1}$, 
we obtain  
\[
  (E - (1 - \sigma) \tau A) (E - (1 - \sigma) \tau A^*) \leq
  (E + \sigma \tau A)(E + \sigma \tau A^*) .
\]
It follows that 
\[
  \tau (A + A^*) + (\sigma^2 - (1-\sigma )^2) \tau^2 A A^* \geq 0.
\]
This inequality holds for the nonnegative operators $A$  with $\sigma \geq 0.5$.
In view of (\ref{12})   we have from (\ref{11})  required estimate  (\ref{8}).
\end{proof}

\section{Operators of decomposition}
\label{s-3}

To better understand the formal structure of the domain decomposition operators, we give a typical example.  
We consider a model unsteady convection-diffusion problem
with a constant (independent of time, but depending on the points
of a computational domain) coefficients of diffusion 
and convection.  
Convective transport is written in the so-called (see, for example, \cite{SamVabConvection})  symmetric form.  
Let in a bounded domain $\Omega$  an unknown function $u(\mathbf{x},t)$
satisfies the following equation
\[
   \frac{\partial u}{\partial t} + \frac{1}{2}
   \sum_{\alpha =1}^{m} \left ( 
   v_\alpha ({\bf x}) \frac{\partial u}{\partial x_\alpha}\ +
   \frac{\partial }{\partial x_\alpha} (v_\alpha ({\bf x}) u) \right ) 
\]
\begin{equation}\label{13}
   - \sum_{\alpha =1}^{m}
   \frac{\partial }{\partial x_\alpha} 
   \left ( k({\bf x})  \frac{\partial u}{\partial x_\alpha} \right ) = f({\bf x},t),
   \quad {\bf x}\in \Omega,
   \quad 0 < t < T,
\end{equation}
where  $k(\mathbf{x}) \geq \kappa > 0, \  {\bf x}\in \Omega$.
Supplement equation (\ref{13}) with the homogeneous Dirichlet boundary conditions  
\begin{equation}\label{14}
   u({\bf x},t) = 0,
   \quad {\bf x}\in \partial \Omega,
   \quad 0 < t < T.
\end{equation}
In addition, we define the initial condition  
\begin{equation}\label{15}
   u({\bf x},0) = u^0({\bf x}),
   \quad {\bf x}\in \Omega.
\end{equation}

Let us consider a set of functions $u({\bf x},t)$ 
satisfying boundary conditions (\ref{14}). 
We write the unsteady convection-diffusion problem
in the form of differential-operator equation
\begin{equation}\label{16}
  \frac {du}{dt} + {\cal A} u = f(t),
  \quad 0 < t < T.
\quad t > 0
\end{equation}
We consider the Cauchy problem for evolutionary equation  (\ref{16}):
\begin{equation}\label{17}
  u(0) = u^0 .
\end{equation}
Operators of diffusive and convective transport are treated separately, so that in  (\ref{16})
\begin{equation}\label{18}
  {\cal A} = {\cal C} + {\cal D} .
\end{equation}
For the diffusion operator we set  
\[
  {\cal D} u = 
   - \sum_{\alpha =1}^{m}
   \frac{\partial }{\partial x_\alpha} 
   \left ( k({\bf x})  \frac{\partial u}{\partial x_\alpha} \right ) .
\]
On  set of functions (\ref{14}) in ${\cal H} =  {\cal L}_2 (\Omega)$
diffusive transport operator ${\cal D}$  is self-adjoint and positive definite:  
\begin{equation}\label{19}
  {\cal D} = {\cal D}^* \ge \kappa \delta {\cal E},
  \quad \delta  = \delta(\Omega) > 0,
\end{equation}
where ${\cal E}$ is the identity operator in  ${\cal H}$.

Convective transport operator ${\cal C}$ is defined by the expression  
\[
  {\cal C} u = \frac{1}{2}
   \sum_{\alpha =1}^{m} \left ( 
   v_\alpha ({\bf x}) \frac{\partial u}{\partial x_\alpha}\ +
   \frac{\partial }{\partial x_\alpha} (v_\alpha ({\bf x}) u) \right ) . 
\]
For any $v_\alpha ({\bf x})$ the operator ${\cal C}$ is skew-symmetric  in ${\cal H}$:
\begin{equation}\label{20}
  {\cal C} = - {\cal C}^*.
\end{equation}
Taking into account representation  (\ref{18}), it follows from (\ref{19}), (\ref{20}) that ${\cal A} > 0$ in ${\cal H}$.

Domain decomposition schemes will be constructed via the partition of unit for the computational domain  $\Omega$. 
Let domain $\Omega$  consists of $p$  separate subdomains  
\[
  \Omega =\Omega _1\cup \Omega _2\cup ...\cup \Omega _p.
\]
Individual subdomains can overlap.  
With separate subdomain $\Omega_{\alpha}, \ \alpha = 1,2,...,p$
we associate function $\eta_{\alpha}(\mathbf{x}), \ \alpha = 1,2,...,p$ such that  
\begin{equation}\label{21}
  \eta_{\alpha}(\mathbf{x}) = \left \{
   \begin{array}{cc}
     > 0, &  \mathbf{x} \in \Omega_{\alpha},\\
     0, &  \mathbf{x} \notin  \Omega_{\alpha}, \\
   \end{array}
  \right .
  \quad \alpha = 1,2,...,p ,  
\end{equation}
where  
\begin{equation}\label{22}
  \sum_{\alpha =1}^{p} \eta_{\alpha}(\mathbf{x}) = 1,
  \quad \mathbf{x} \in \Omega .
\end{equation}

In view of (\ref{21}), (\ref{22})  we obtain from  (\ref{18}) the following representation  
\begin{equation}\label{23}
  {\cal A} = \sum_{\alpha =1}^{p} {\cal A}_{\alpha},
  \quad {\cal A}_{\alpha} = {\cal C}_{\alpha} + {\cal D}_{\alpha} ,
  \quad \alpha = 1,2,...,p , 
\end{equation}
where  
\[
  {\cal D}_{\alpha} u = 
   - \sum_{\alpha =1}^{m}
   \frac{\partial }{\partial x_\alpha} 
   \left ( k({\bf x}) \eta_{\alpha}(\mathbf{x}) \frac{\partial u}{\partial x_\alpha} \right ) ,
\]
\[
  {\cal C}_{\alpha} u = \frac{1}{2}
   \sum_{\alpha =1}^{m} \left ( 
   v_\alpha ({\bf x}) \eta_{\alpha}(\mathbf{x}) \frac{\partial u}{\partial x_\alpha}\ +
   \frac{\partial }{\partial x_\alpha} (v_\alpha ({\bf x}) \eta_{\alpha}(\mathbf{x})u) \right ) . 
\]
Similarly (\ref{19}), (\ref{20}),  we have  
\begin{equation}\label{24}
  {\cal D}_{\alpha} = {\cal D}^*_{\alpha} \geq 0,
  \quad {\cal C}_{\alpha} = - {\cal C}^*_{\alpha},
  \quad \alpha = 1,2,...,p . 
\end{equation}
In view of (\ref{24}) in splitting  (\ref{23}) the following property holds  
\begin{equation}\label{25}
  {\cal A}_{\alpha}  \geq 0,
  \quad \alpha = 1,2,...,p ,
\end{equation}
and the self-adjoint part of operator ${\cal A}$  is splitted into the sum of nonnegative self-adjoint operators, 
whereas  the skew-symmetric one -- into the sum of skew-symmetric operators. 

It is convenient to represent diffusive transport operator ${\cal D}$ as follows
\begin{equation}\label{26}
  {\cal D} = {\cal G}^* {\cal G},
  \quad {\cal G} = k^{1/2} \grad,
  \quad {\cal G}^* = -  \div k^{1/2} ,
\end{equation}
with  ${\cal G}: {\cal H} \rightarrow \widetilde{{\cal H}}$,
where  $\widetilde{{\cal H}} = ({\cal L}_2(\Omega))^p$ is the corresponding Hilbert space of vector functions.  
From this structure, for ${\cal D}_{\alpha}, \ \alpha = 1,2, ...,p$  we obtain  
\begin{equation}\label{27}
  {\cal D}_{\alpha} = {\cal G}^* \eta_{\alpha} {\cal G},
  \quad \alpha = 1,2,...,p .
\end{equation}
Similarly, ${\cal C}_{\alpha}, \ \alpha = 1,2, ...,p$ can be represented as
\begin{equation}\label{28}
  {\cal C}_{\alpha} =  
  \frac{1}{2} ( \eta_{\alpha} {\cal C} + {\cal C}  \eta_{\alpha}),
  \quad \alpha = 1,2,...,p .
\end{equation}
Representations (\ref{27}), (\ref{28}) for operators of diffusive and convective transport
demonstrate us clearly the structure of operators in individual subdomains in the splitting based 
on (\ref{21}), (\ref{22})  and  allow to verify fulfilment of  (\ref{24}).
A similar consideration can be given for the operator of problem (\ref{2}), (\ref{3}).

Let us divide the operator $A$ into the self-adjoint and skew-symmetric parts:  
\begin{equation}\label{29}
  A = C + D,
  \quad C = \frac{1}{2} (A-A^*),
  \quad D = \frac{1}{2} (A+A^*).
\end{equation}
For the nonnegative operator $D$  the following representation  holds
\begin{equation}\label{30}
  D = G^* G ,
\end{equation}
where  $G: H \rightarrow \widetilde{H}$.
For the partition of unit of the computational domain we consider the corresponding additive representation of 
unit operators $E$  and $\widetilde{E}$
in spaces $H$ and $\widetilde{H}$, respectively.  
Let  
\begin{equation}\label{31}
  \sum_{\alpha =1}^{p} \chi_{\alpha} = E,
  \quad \chi_{\alpha} \geq 0,
  \quad \alpha = 1,2,...,p ,
\end{equation}
\begin{equation}\label{32}
  \sum_{\alpha =1}^{p} \widetilde{\chi}_{\alpha} = \widetilde{E},
  \quad \widetilde{\chi}_{\alpha} \geq 0,
  \quad \alpha = 1,2,...,p .
\end{equation}

By analogy with (\ref{23})--(\ref{25}), we use the splitting  
\begin{equation}\label{33}
  A = \sum_{\alpha =1}^{p} A_{\alpha},
  \quad A_{\alpha} \geq 0 ,
  \quad \alpha = 1,2,...,p , 
\end{equation}
where  
\begin{equation}\label{34}
  A_{\alpha} = C_{\alpha} + D_{\alpha} ,
  \quad D_{\alpha} = D^*_{\alpha} \geq 0 ,
  \quad C_{\alpha} = - C^*_{\alpha} ,
  \quad \alpha = 1,2,...,p . 
\end{equation}
On the basis of (\ref{32})  for the terms of the self-adjoint part of the operator  $A$ we set  
\begin{equation}\label{35}
  D_{\alpha} = G^* \widetilde{\chi}_{\alpha} G,
  \quad \alpha = 1,2,...,p .
\end{equation}
Decomposition of the skew-symmetric part is based on (\ref{31}):
\begin{equation}\label{36}
  C_{\alpha} = 
  \frac{1}{2} (  \chi_{\alpha} C + C  \chi_{\alpha}),
  \quad \alpha = 1,2,...,p .
\end{equation}
Such an additive representation is a discrete analog of (\ref{27}), (\ref{28}) 
and is interpreted as the corresponding version of the domain decomposition.  

\section{Regularized domain decomposition schemes}
\label{s-4}

For the approximate solution of the Cauchy problem for equation (\ref{2}), (\ref{3})  under condition   (\ref{33})
we apply different splitting schemes.  
Transition to the new time level is based on solving $p$ separate subproblems 
with individual operators $A_{\alpha}, \ \alpha = 1,2, ...,p$.
Taking into account the structure of the operators (see (\ref{34})--(\ref{36})) we can say
that these splitting schemes are regionally-additive and iteration-free.  

Currently, the principle of regularization of difference schemes is considered as the basic methodological principle
for improving difference schemes   \cite{0971.65076}.      
The construction of unconditionally stable additive difference schemes  \cite{0963.65091}
via the principle of regularization  will be implemented in the following way.  

\begin{enumerate}
  \item  For the initial problem there is constructed some simple difference scheme (the producing difference scheme). 
  This scheme does not satisfy the necessary properties.  For example, in constructing additive schemes 
  the producing scheme is not splitting one, or it is conditionally stable or even absolutely unstable.  

  \item The difference scheme is written in the form for which the stability conditions are known.  

  \item  Quality of the scheme (its stability) is improved via perturbations of operators of the difference scheme 
  with preserving possibility of its computational implementation as an additive scheme.  
\end{enumerate}

Concerning to problem (\ref{2}), (\ref{3})  it is natural to choose as the producing scheme the following simple explicit scheme
\begin{equation}\label{37}
  \frac{y^{n+1} - y^{n}}{\tau }
  +  A y^{n} = \varphi^n,
  \quad n = 0,1, ..., N-1,
\end{equation}
which is supplemented by initial conditions (\ref{7}).
Stability of scheme  (\ref{37}) is provided (see the proof of Theorem~\ref{t-1}) by fulfillment of inequality  
\begin{equation}\label{38}
  A + A^* - \tau A A^* \geq 0.
\end{equation}
Inequality (\ref{38}) with $D > 0$ imposes appropriate restrictions on the time step, 
i.e.  scheme (\ref{29}), (\ref{37})  is conditionally stable.  
Note also that if $D = 0$ than  scheme (\ref{29}), (\ref{37}) is absolutely unstable. 
Taking into account splitting (\ref{33}), 
we refer this scheme to the class of additive schemes.  

To construct additive schemes, we can take more general scheme (\ref{6}), (\ref{7}) as a producing one.
It is unconditionally stable at  $\sigma \geq 0.5$. 
In this case the perturbation of scheme operators is oriented  only to receive the additive schemes 
preserving the property of unconditional stability.  

Regularization of a difference scheme in order to improve the stability restriction  
(construction of a splitting scheme) can be performed via some perturbation of the operator $A$. 
The second possibility is related to perturbation of the operator at the difference derivative in time 
(for our scheme (\ ref {37}) it is operator $E$).
In constructing additive schemes, it is convenient to operate with the transition operator   $S$,  
rewriting producing scheme  (\ref{37})  as  follows
\begin{equation}\label{39}
  y^{n+1} = S y^{n} + \tau \varphi^n ,
  \quad n = 0,1, ..., N-1 .
\end{equation}
For  (\ref{37}) we have  
\begin{equation}\label{40}
  S = E - \tau A .
\end{equation}
The regularized scheme is based on perturbation of the operator $S$ and has the following form 
\begin{equation}\label{41}
  y^{n+1} = \widetilde{S} y^{n} + \tau \varphi^n ,
  \quad n = 0,1, ..., N-1 .
\end{equation}
Let us consider general restrictions for $\widetilde{S}$.

To preserve the first order approximation, which has generating scheme (\ref{39}), (\ref{40}), 
we subordinate  the selection of $\widetilde{S}$ to the condition  
\begin{equation}\label{42}
  \widetilde{S} = E - \tau A + \mathcal{O} (\tau^2) .
\end{equation}
Stability of scheme  (\ref{41}) in the sense that  estimate  (\ref{8}) holds, is provided by the inequality  
\begin{equation}\label{43}
  \| \widetilde{S} \| \leq 1 .
\end{equation}
In addition, the regularized scheme must be additive, i.e. the transition to the new time 
level is implemented via solving the individual subproblems for operators $A_{\alpha}, \ \alpha = 1,2,...,p$  
in decomposition  (\ref{33}).

The first class of regularized splitting schemes is based on the following additive representation 
for the transition operator of the producing scheme  
\[
  S = \frac{1}{p} \sum_{\alpha =1}^{p} S_{\alpha},
  \quad S_{\alpha} = E - p \tau A_{\alpha},
  \quad \alpha = 1,2,...,p .
\]
The similar additive representation we also use for the transition operator of the regularized scheme  
\begin{equation}\label{44}
  \widetilde{S} = \frac{1}{p} \sum_{\alpha =1}^{p} \widetilde{S}_{\alpha},
  \quad \alpha = 1,2,...,p .  
\end{equation}
Individual terms $\widetilde{S}_{\alpha}, \ \alpha = 1,2,...,p$  are constructed via perturbations of operators 
$A_{\alpha}, \ \alpha = 1,2,...,p$.
By analogy with(\ref{10}) we set  
\begin{equation}\label{45}
  \widetilde{S}_{\alpha} = 
  (E + \sigma p \tau A_{\alpha})^{-1} (E - (1 - \sigma) p \tau A_{\alpha}),
  \quad \alpha = 1,2,...,p .  
\end{equation}
If  $\sigma \geq 0.5$  (see proof of Theorem~\ref{t-1}) we have  
\[
  \|\widetilde{S}_{\alpha} \| \leq 1,
  \quad \alpha = 1,2,...,p . 
\]
In view of (\ref{44})  it provides fulfilment of  stability conditions (\ref{43}). 

Using the representation
\[
  \widetilde{S}_{\alpha} = 
  E - p \tau (E + \sigma p \tau A_{\alpha})^{-1} A_{\alpha},
  \quad \alpha = 1,2,...,p   
\]
we can rewrite regularized additive scheme (\ref{41}), (\ref{44}), (\ref{45}) as follows  
\begin{equation}\label{46}
  \frac{y^{n+1} - y^{n}}{\tau }
  +  \sum_{\alpha =1}^{p} (E + \sigma p \tau A_{\alpha})^{-1} 
  A_{\alpha} y^{n} = \varphi^n,
  \quad n = 0,1, ..., N-1 .
\end{equation}
The comparison with producing scheme  (\ref{33}), (\ref{37})  shows that 
the regularization is provided by perturbation of the operator $A$.
Our consideration results in the following statement.  

\begin{theorem} 
\label{t-2} 
Additive difference scheme(\ref{7}), (\ref{41}), (\ref{44}), (\ref{45})
 is unconditionally stable at $\sigma \geq 0.5$,
and for the numerical solution estimate of stability  (\ref{8}) with respect to the initial data and right-hand side  holds.  
\end{theorem} 

Numerical implementation of scheme (\ref{7}), (\ref{46})  can be conducted as follows. 
Assume 
\[
  y^{n+1} = \frac{1}{p} \sum_{\alpha =1}^{p} y_{\alpha}^{n+1},
  \quad \varphi^{n} = \sum_{\alpha =1}^{p} \varphi_{\alpha}^{n} .
\]
In this case we obtain  
\begin{equation}\label{47}
  \frac{y_{\alpha}^{n+1} - y^{n}}{p \tau }
  +  (E + \sigma p \tau A_{\alpha})^{-1} 
  A_{\alpha} y^{n} = \varphi_{\alpha}^n,
  \quad \alpha = 1,2,...,p .
\end{equation}
for  the individual components of the approximate solution at the new time level
$y_{\alpha}^{n+1}, \ \alpha = 1,2,...,p$.
Scheme (\ref{47})  can be rewritten as  follows
\[
  \frac{y_{\alpha}^{n+1} - y^{n}}{p \tau }
  +  A_{\alpha} y^{n} (\sigma  y_{\alpha}^{n+1} 
  + (1 - \sigma ) y^{n})= 
  (E + \sigma p \tau A_{\alpha}) \varphi_{\alpha}^n .
\]
In this form we can interpret scheme  (\ref{47})
as a variant of the additively-averaged scheme of component-wise splitting  \cite{0963.65091}.

The second class of regularized splitting schemes is based on using not additive (see(\ref{44})) but multiplicative 
representation of the transition operator:  
\begin{equation}\label{48}
  \widetilde{S} = \prod_{\alpha =1}^{p} \widetilde{S}_{\alpha},
  \quad \alpha = 1,2,...,p .  
\end{equation}
Taking into account (\ref{42}), we have 
\[
  S = \prod_{\alpha =1}^{p} S_{\alpha} + \mathcal{O} (\tau^2),
  \quad S_{\alpha} = E - \tau A_{\alpha},
  \quad \alpha = 1,2,...,p .
\]
Similarly (\ref{45}),  we set  
\begin{equation}\label{49}
  \widetilde{S}_{\alpha} = 
  (E + \sigma \tau A_{\alpha})^{-1} (E - (1 - \sigma) \tau A_{\alpha}),
  \quad \alpha = 1,2,...,p .  
\end{equation}
Under the standard restrictions $\sigma \geq 0.5$ regularized scheme(\ref{41}), (\ref{48}), (\ref{49}) is stable.  

\begin{theorem} 
\label{t-3} 
Additive difference scheme  (\ref{7}), (\ref{41}), (\ref{48}), (\ref{49}) is unconditionally stable 
at  $\sigma \geq 0.5$, and  estimate (\ref{8}) of stability with respect to
the initial data and right-hand side is valid for the difference solution.
\end{theorem} 

We present now some possible implementation of the constructed regularized scheme.  
Let us introduce auxiliary quantities $y^{n+\alpha/p}, \ \alpha = 1,2,...,p$  
and taking into account (\ref{41}), (\ref{48}), we define them from the equations  
\[
  y^{n+\alpha/p} = \widetilde{S}_{\alpha} y^{n+(\alpha-1)/p} ,
  \quad \alpha = 1,2,...,p-1, 
\] 
\begin{equation}\label{50}
  y^{n+1} = \widetilde{S}_{p} y^{n+(p-1)/p} + 
  \tau \varphi^n  .
\end{equation}
Similar to  (\ref{47})  we obtain  from (\ref{50})
\begin{equation}\label{51}
  \frac{y^{n+\alpha/p} - y^{n+(\alpha-1)/p}}{\tau }
  +  (E + \sigma \tau A_{\alpha})^{-1} 
  A_{\alpha} y^{n+(\alpha-1)/p} = \varphi_{\alpha}^n,
\end{equation}
where  
\[
  \varphi_{\alpha}^n = \left \{
  \begin{array}{ll}
   0, &  \alpha = 1,2,...,p-1, \\
   \varphi^n ,  & \alpha = p .   \\
  \end{array}
  \right .
\]
Rewrite  scheme(\ref{51}) as follows 
\begin{equation}\label{52}
  \frac{y^{n+\alpha/p} - y^{n+(\alpha-1)/p}}{\tau }
  + A_{\alpha} (\sigma y^{n+\alpha/p} + (1-\sigma)y^{n+(\alpha-1)/p}) =
  \widetilde{\varphi}_{\alpha}^n,
\end{equation}
where  
\[
  \widetilde{\varphi}_{\alpha}^n = 
  (E + \sigma \tau A_{\alpha}) \varphi_{\alpha}^n ,
  \quad \alpha = 1,2,...,p .
\]
Scheme (\ref{52}) is a special variant of the standard component-wise splitting scheme \cite{Marchuk:1990:SAD,0971.65076,0209.47103}.
But unlike these schemes of summarized  approximation we constructed here the regularized schemes of full approximation.  
Regularized schemes (\ref{41}), (\ref{44}), (\ref{45}),  constructed using additive representation 
(\ref{44})  for the transition operator, are more suitable for parallel computations in compare with
regularized schemes (\ref{41}), (\ref{48}), (\ref{49}) which are based on multiplicative representation  (\ref{48}).

\section{Vector schemes of domain decomposition}
\label{s-5}

Difference schemes for unsteady problems can often be treated as appropriate iterative 
methods for the approximate solution of stationary problems. 
Great opportunities in this direction  provide the vector additive schemes  \cite{0712.65089,vabishchevich1996vector}.

Instead of the single unknown $u(t)$ we consider $p$ unknowns  $u_{\alpha}, \  \alpha = 1,2,...,p$, 
which are determined from the system  
\begin{equation}\label{53}
  \frac{d  u_{\alpha}}{d  t} + 
  \sum_{\beta = 1}^{p} A_{\beta} u_{\beta} = f(t),    
  \quad \alpha = 1,2,...,p ,
  \quad 0 < t \leq T . 
\end{equation}
The system of equations  (\ref{53}) is supplemented with the initial conditions  
\begin{equation}\label{54}
  u_{\alpha}(0) =  u^0,
  \quad \alpha = 1,2,...,p ,
\end{equation}
which follow from (\ref{2}).
Obviously, each function is a solution of problem (\ref{2}), (\ref{3}), (\ref{33}).
The approximate solution of (\ref{2}), (\ref{3}), (\ref{33})  
will be constructed on the basis of one or another difference scheme for vector problem (\ref{53}), (\ref{54}).

To solve problem (\ref{53}), (\ref{54}), we use the following two-level scheme  
\[
  \frac{y_{\alpha}^{n+1} - y_{\alpha}^{n}}{\tau }
  +  \sum_{\beta = 1}^{\alpha }A_{\beta}y_{\beta }^{n+1}
  +  \sum_{\beta = \alpha + 1}^{p} A_{\beta}y_{\beta }^{n}
  = \varphi^n,
\]
\begin{equation}\label{55}
  \quad \alpha = 1,2,...,p ,
  \quad n = 0,1, ..., N-1 .  
\end{equation}
For this difference scheme we use the initial conditions  
\begin{equation}\label{56}
  y_{\alpha}(0) =  u^0,
  \quad \alpha = 1,2,...,p .
\end{equation}
Numerical implementation of this scheme is based on
the successive inversion of operators  $E + \tau A_{\alpha}, \ \alpha = 1,2,...,p$.

\begin{theorem} 
\label{t-4} 
Vector additive difference scheme (\ref{33}), (\ref{55}), (\ref{56})
is unconditionally stable, and for the components of the difference solution the following stability 
estimate with respect to the initial data and right-hand side
\[
  \|y_{\alpha}^{n+1}\| \leq \|y_{\alpha}^{n}\| + 
  \tau  \| \varphi^{0} - A u^0 \| +  
  \tau \sum_{k=1}^{n} \tau 
  \left \| \frac{\varphi^{k} - \varphi^{k-1}}{\tau } \right \| ,
\]
\begin{equation}\label{57}
  \alpha = 1,2,...,p ,
  \quad n = 0,1, ..., N-1 ,
\end{equation}
 is valid.
\end{theorem} 

\begin{proof}
To study  vector scheme (\ref{55}), (\ref{56}), it is convenient to use the approach from the work 
\cite{samarskii1998stability}.
Subtracting the  $\alpha$-th  equation for $y_{\alpha+1}^{n}$  from the $\alpha + 1$-th  equations of system (\ref{55})
for  $y_{\alpha+1}^{n+1}$, we get
\begin{equation}\label{58}
  (E + \tau A_{\alpha + 1}) 
  \frac{y_{\alpha+1}^{n+1} - y_{\alpha+1}^{n}}{\tau } =
  \frac{y_{\alpha}^{n+1} - y_{\alpha}^{n}}{\tau },
  \quad \alpha = 1,2, ..., p-1.
\end{equation}
Similarly, considering the equations for $y_{1}^{n+1}$ and $y_{p}^{n}$, we obtain  
\begin{equation}\label{59}
  (E + \tau A_{1}) 
  \frac{y_{1}^{n+1} - y_{1}^{n}}{\tau } =
  \frac{y_{p}^{n} - y_{p}^{n-1}}{\tau } +
  \tau \frac{\varphi^{n} - \varphi^{n-1}}{\tau } .
\end{equation}

Taking into account nonnegativity of  operators $A_{\alpha}, \ \alpha = 1,2, ..., p$,
from (\ref{58}) we obtain  
\begin{equation}\label{60}
  \left \| \frac{y_{\alpha+1}^{n+1} - y_{\alpha+1}^{n}}{\tau } \right \|
  \leq 
  \left \| \frac{y_{\alpha}^{n+1} - y_{\alpha}^{n}}{\tau } \right \|,
  \quad \alpha = 1,2, ..., p-1.
\end{equation}
Similarly, from (\ref{59}) we have  
\begin{equation}\label{61}
  \left \| \frac{y_{1}^{n+1} - y_{1}^{n}}{\tau } \right \|
  \leq 
  \left \| \frac{y_{p}^{n} - y_{p}^{n-1}}{\tau } \right \|+
  \tau \left \| \frac{\varphi^{n} - \varphi^{n-1}}{\tau } \right \| .
\end{equation}
From (\ref{60}), (\ref{61}), we derive at each time level the following estimate
\[
  \left \| \frac{y_{\alpha}^{n+1} - y_{\alpha}^{n}}{\tau } \right \|
  \leq 
  \left \| \frac{y_{\alpha}^{n} - y_{\alpha}^{n-1}}{\tau } \right \|+
  \tau \left \| \frac{\varphi^{n} - \varphi^{n-1}}{\tau } \right \| ,
\]
\begin{equation}\label{62}
  \alpha = 1,2, ..., p,
  \quad n = 1,2, ..., N-1 .
\end{equation}
From  (\ref{62}) we get
\[
  \left \| \frac{y_{\alpha}^{n+1} - y_{\alpha}^{n}}{\tau } \right \|
  \leq 
  \left \| \frac{y_{\alpha}^{1} - y_{\alpha}^{0}}{\tau } \right \|+
  \sum_{k=1}^{n} \tau 
  \left \| \frac{\varphi^{k} - \varphi^{k-1}}{\tau } \right \| ,
\]
\begin{equation}\label{63}
  \alpha = 1,2, ..., p,
  \quad n = 1,2, ..., N-1 .
\end{equation}
From (\ref{55}) with  $\alpha = 1$, taking into account splitting (\ref{33})
and initial conditions (\ref{56})), we obtain  
\[
  \left \| \frac{y_{1}^{1} - y_{1}^{0}}{\tau } \right \| \leq 
  \| \varphi^{0} - A u^0 \| .  
\]
In view of (\ref{60}) we can rewrite inequality (\ref{63})  as  follows
\[
  \left \| \frac{y_{\alpha}^{n+1} - y_{\alpha}^{n}}{\tau } \right \|
  \leq 
  \| \varphi^{0} - A u^0 \| +  
  \sum_{k=1}^{n} \tau 
  \left \| \frac{\varphi^{k} - \varphi^{k-1}}{\tau } \right \| ,
\]
\begin{equation}\label{64}
  \alpha = 1,2, ..., p,
  \quad n = 1,2, ..., N-1 .
\end{equation}
Taking into account the obvious inequality  
\[
  \|y_{\alpha}^{n+1}\| \leq \|y_{\alpha}^{n}\| + 
  \tau \left \| \frac{y_{\alpha}^{n+1} - y_{\alpha}^{n}}{\tau } \right \| ,
  \quad \alpha = 1,2,...,p ,
\]
we obtain from (\ref{64}) required estimate  (\ref{57}).
\end{proof}

We emphasize that the above stability estimates(\ref{57}) are received 
for each individual component $y_{\alpha}^{n+1}, \ \alpha = 1,2,...,p$. 
Each of them or their linear combination  
\[
  y^{n+1} = \sum_{\alpha =1}^{p} c_{\alpha} y_{\alpha}^{n+1},
  \quad c_{\alpha} = \const \geq 0,
  \quad \alpha = 1,2, ..., p
\]
can be treated as an approximate solution of our problem  
(\ref{2}), (\ref{3}), (\ref{33}) at time moment $t=t^{n+1}$.

\section{Model problem}
\label{s-6}

To illustrate possibilities of the constructed here domain decomposition schemes, 
let us consider the simplest boundary value problem for the parabolic equation.  
We consider the problem in a rectangle  
\[
  \Omega = \{ \ \mathbf{x} \ | \ \mathbf{x} = (x_1, x_2), 
  \ 0 < x_{\alpha} < l_{\alpha}, \ \alpha =1,2 \}.
\]
In $\Omega $  the following boundary  value problem 
\begin{equation}\label{65}
   \frac{\partial u}{\partial t} = 
   \sum_{\alpha =1}^{2} \frac{\partial^2 u}{\partial x^2_\alpha} , 
   \quad {\bf x}\in \Omega,
   \quad 0 < t < T,
\end{equation}
\begin{equation}\label{66}
   u({\bf x},t) = 0,
   \quad {\bf x}\in \partial \Omega,
   \quad 0 < t < T,
\end{equation}
\begin{equation}\label{67}
   u({\bf x},0) = u^0({\bf x}),
   \quad {\bf x}\in \Omega
\end{equation}
is solved.  

The approximate solution is searched at the nodes of a uniform rectangular grid  in  $\Omega$:
 \[
   \bar{\omega} = \{ \mathbf{x} \ | \ \mathbf{x} = (x_1, x_2),
   \quad x_\alpha = i_\alpha h_\alpha,
   \quad i_\alpha = 0,1,...,N_\alpha,
   \quad N_\alpha h_\alpha = l_\alpha\} 
\]
and let  $\omega$ be the set of internal nodes
($\bar{\omega} = \omega \cup \partial \omega$). 
For grid functions $y(\mathbf{x}) = 0, \ \mathbf{x} \in \partial \omega$
we define the Hilbert space  $H = L_2({\omega})$  with the scalar product and norm  
\[
  (y,w) \equiv \sum_{{\bf x} \in \omega}
  y({\bf x}) w({\bf x}) h_1 h_2,
  \quad \|y\| \equiv (y,y)^{1/2} .
\]

Approximating problem (\ref{65}), (\ref{66})  in space,
we obtain the differential-difference equation 
\begin{equation}\label{68}
  \frac{d y}{d t} + A y = 0,
  \quad \mathbf{x} \in \omega,
   \quad 0 < t < T,
\end{equation}
where  
\[
  A y =
  - \frac{1}{h_1^2} (y(x_1+h_1,x_2) - 2y(x_1,x_2) - y(x_1-h_1,x_2))
\]
\begin{equation}\label{69}
  - \frac{1}{h_2^2} (y(x_1,x_2+h_2) - 2y(x_1,x_2) - y(x_1,x_2-h_2)),
  \quad \mathbf{x} \in \omega .
\end{equation}
In the space $H$ the operator  $A$  is self-adjoint and positive definite 
 \cite{0971.65076,samarskii1989numerical}:
\begin{equation}\label{70}
  A = A^* \geq (\delta_1+\delta_2) E,
  \quad \delta_{\alpha} = 
  \frac{4}{h^2_{\alpha}} \sin^2 \frac{\pi h_{\alpha}}{2 l_{\alpha}} ,
  \quad \alpha = 1,2.
\end{equation}
Taking into account (\ref{67}), we supplement equation (\ref{69}) 
with the initial condition  
\begin{equation}\label{71}
   y({\bf x},0) = u^0({\bf x}),
   \quad {\bf x}\in \omega .
\end{equation}

For simplicity, operators of domain decomposition in the investigated problem
(\ref{68})--(\ref{71})  will be constructed without the explicit definition of operators
 $G$  and  $G$  as well as space $\widetilde{H}$, 
focusing on decomposition (\ref{21}), (\ref{22}).
We set  
\[
  A_{\alpha} y =
  - \frac{1}{h_1^2} \eta_{\alpha}(x_1+0.5h_1,x_2)
  (y(x_1+h_1,x_2) - y(x_1,x_2))
\]
\[
  + \frac{1}{h_1^2} \eta_{\alpha}(x_1-0.5h_1,x_2)
  (y(x_1,x_2) - y(x_1-h_1,x_2)) 
\]
\[
  - \frac{1}{h_2^2} \eta_{\alpha}(x_1,x_2+0.5h_2)
  (y(x_1,x_2+h_2) - y(x_1,x_2))
\]
\begin{equation}\label{72}
  + \frac{1}{h_2^2} \eta_{\alpha}(x_1,x_2-0.5h_2)
  (y(x_1,x_2) - y(x_1,x_2-h_2)) ,
  \quad \alpha =1,2,...,p .
\end{equation}
In view of (\ref{21}), (\ref{22}) we have 
\begin{equation}\label{73}
  A = \sum_{\alpha =1}^{p}  A_{\alpha},
  \quad A_{\alpha} = A^*_{\alpha},
  \quad \alpha =1,2,...,p .
\end{equation}
Thus, we consider the class of additive schemes (\ref{33}).  

\begin{figure}[h]
  \begin{tikzpicture}
    \fill [black!15] (0,0) rectangle +(8.,8.);
    \draw (0,0) rectangle +(8.,8.);
	\foreach \x in {0,...,1}
      \fill[black!30] (4*\x+2,0) rectangle +(2.,8.);
	\foreach \x in {0,...,2}
      \fill[black!45] (2*\x+1.9,0) rectangle +(0.2,8.);
  \end{tikzpicture}
  \caption{Domain decomposition  }
\label{f-1}
\end{figure}
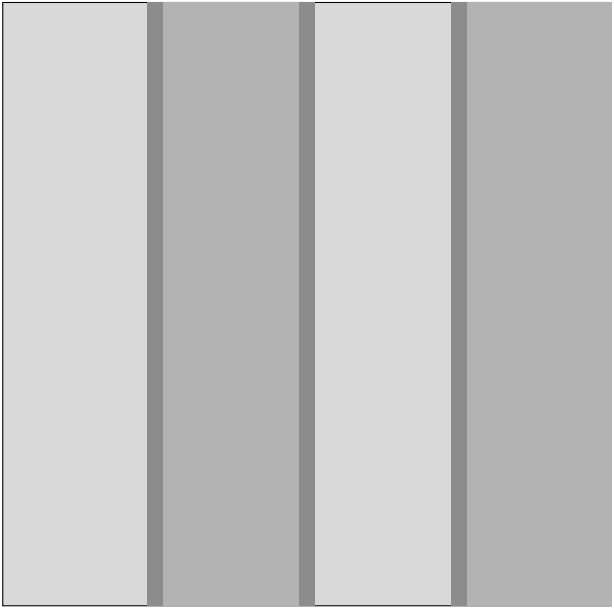

Numerical calculations  for problem  (\ref{65})--(\ref{67}) are performed in the unit square
 ($l_1 = l_2 = 1$) where the solution has the form  
\begin{equation}\label{74}
  u(\mathbf{x},t) = 
  \sin(n_1 \pi x_1) \sin(n_2 \pi x_2) 
  \exp (- \pi^2(n_1^2  + n_2^2) t)
\end{equation}
for a natural $ n_1 $ and $ n_2 $.
For this solution we set the corresponding initial conditions (\ref{67}).
Decomposition is performed with respect to one of two variables into four subdomains (see Fig.~\ref{f-1}) 
with overlapping.  
Disconnected subdomains can be considered as some single subdomain and the decomposition in Fig.~\ref{f-1} 
can be treated as the decomposition into two subdomains described via two functions:  
$\eta_{\alpha} = \eta_{\alpha}(x_1), \ \alpha =1,2$.

For problems of type  (\ref{65})--(\ref{67}) 
two cases of domain decomposition are highlighted: decomposition with and without overlapping  of subdomains.  
Methods without overlapping  of subdomains are connected with the explicit formulation 
of certain conditions at the common boundaries.  
In our case a special problem at the interfaces is not formulated, 
but for algorithms without  overlapping we can derive the corresponding exchange boundary conditions.   

For the domain decomposition methods the fundamental issue is exchange of calculation 
data between different subdomains.  
Standard explicit schemes can be used.  
In this case the domain decomposition can be associated with separate subsets of grid 
nodes:$\omega_{\alpha}, \ \alpha =1,2$, where $\omega = \omega_1 \cup  \omega_2$.
In the case of (\ref{65})--(\ref{67}) (the seven point stencil in space), 
the transition to the new  time level  via the explicit scheme 
for finding the approximate solution on grid $\omega_{\alpha}, \ \alpha =1,2$
is performed using the solution values at nodes adjacent to the interface.
We need to transfer the data of size $\sim \partial \omega_{\alpha},\ \alpha =1,2$.
For the approximate solution of problem  (\ref{68})--(\ref{71})
we can consider two possibilities with minimum overlapping of subdomains.  
The first employs the domain decomposition with interfaces at integer nodes ---
the boundary nodes belong to several subdomains (two in our case of decomposition with respect to one variable).   
The second possibility is realized when the boundary of subdomains passes through the half-integer nodes of the corresponding variable.  

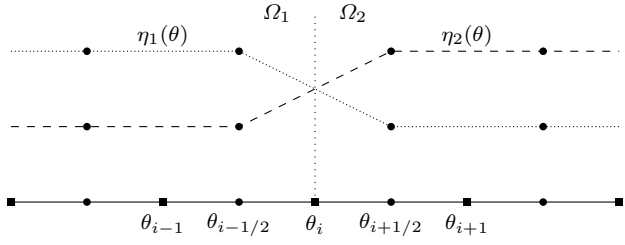
\begin{figure}[h]
  \begin{tikzpicture}
    \draw (-1,0) -- (7,0);
  	\foreach \x in {0,...,3}
	  \fill [black] (2*\x,0) circle (0.05);
  	\foreach \x in {0,...,3}
	  \fill [black] (2*\x,2) circle (0.05);
  	\foreach \x in {0,...,3}
	  \fill [black] (2*\x,1) circle (0.05);
  	\foreach \x in {0,...,4}
	  \fill [black] (2*\x-1-0.05,-0.05) rectangle +(0.1,0.1);
    \draw[densely dotted] (-1,2) -- (2,2) -- (4,1) -- (7,1);
    \draw[dashed] (-1,1) -- (2,1) -- (4,2) -- (7,2);
    \draw[dotted] (3,0) -- (3,2.5);
	\draw(3,-0.3) node {$\theta_i$};
	\draw(1,-0.3) node {$\theta_{i-1}$};
	\draw(5,-0.3) node {$\theta_{i+1}$};
	\draw(2,-0.3) node {$\theta_{i-1/2}$};
	\draw(4,-0.3) node {$\theta_{i+1/2}$};
	\draw(2.5,2.5) node {$\Omega_1$};
	\draw(3.5,2.5) node {$\Omega_2$};
	\draw(1,2.2) node {$\eta_1(\theta)$};
	\draw(5,2.2) node {$\eta_2(\theta)$};
  \end{tikzpicture}
  \caption{Decomposition through integer nodes}
\label{f-2}
\end{figure}

The variant of domain decomposition with boundaries through integer nodes is shown in Fig.~\ref{f-2}.
Assume that the decomposition is carried out in spatial variable $x_1$, i.e.  $\theta = x_1$. 
The boundary of subdomains here passes through the node $\theta = \theta_i$. 
Thus, for this decomposition operators (\ref{72}) take the form  
\[
  A_1 y =
  \frac{1}{h_1^2} (y(x_1,x_2) - y(x_1-h_1,x_2)) 
\]
\[
  - \frac{1}{2h_2^2} (y(x_1,x_2+h_2) - 2y(x_1,x_2) - y(x_1,x_2-h_2)),
\]
\[
  A_{2} y =
  - \frac{1}{h_1^2} (y(x_1+h_1,x_2) - y(x_1,x_2))
\]
\[
  - \frac{1}{2h_2^2} (y(x_1,x_2+h_2) - 2y(x_1,x_2) - y(x_1,x_2-h_2)),
  \quad x_1 = \theta_i .
\]
This decomposition can be associated with using the Neumann boundary conditions as the exchange boundary conditions.  
The relationship between the individual subdomains is minimal and requires exchange of data at $\theta = \theta_i$. 
This case can be identified by the operators of decomposition (\ref{32})
as follows: 
\begin{equation}\label{75}
  R(\widetilde{\chi}_{\alpha}) = [0,1],
  \quad \alpha = 1,2,...,p .
\end{equation}
The values of $\eta_{\alpha}(x_1\pm 0.5h_1,x_2)$,
$\eta_{\alpha}(x_1,x_2\pm 0.5h_1), \ \alpha = 1,2$
for (\ref{72}), (\ref{74}) is equal to 0 or 1.  

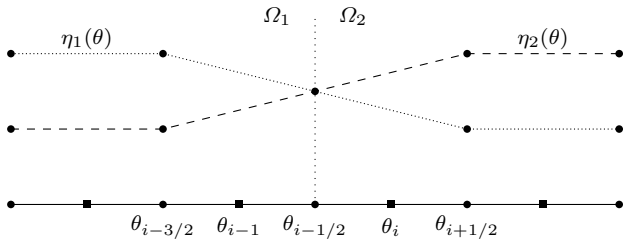
\begin{figure}[h]
  \begin{tikzpicture}
    \draw (-1,0) -- (7,0);
  	\foreach \x in {0,...,3}
	  \fill [black] (2*\x-0.05,-0.05) rectangle +(0.1,0.1);
  	\foreach \x in {0,...,4}
	  \fill [black] (2*\x-1,0) circle (0.05);
  	\foreach \x in {0,...,1}
	  \fill [black] (2*\x-1,2) circle (0.05);
  	\foreach \x in {0,...,1}
	  \fill [black] (2*\x-1,1) circle (0.05);
  	\foreach \x in {0,...,1}
	  \fill [black] (5+2*\x,2) circle (0.05);
  	\foreach \x in {0,...,1}
	  \fill [black] (5+2*\x,1) circle (0.05);
	\fill [black] (3,1.5) circle (0.05);
    \draw[densely dotted] (-1,2) -- (1,2) -- (5,1) -- (7,1);
    \draw[dashed] (-1,1) -- (1,1) -- (5,2) -- (7,2);
    \draw[dotted] (3,0) -- (3,2.5);
	\draw(3,-0.3) node {$\theta_{i-1/2}$};
	\draw(1,-0.3) node {$\theta_{i-3/2}$};
	\draw(5,-0.3) node {$\theta_{i+1/2}$};
	\draw(2,-0.3) node {$\theta_{i-1}$};
	\draw(4,-0.3) node {$\theta_{i}$};
	\draw(2.5,2.5) node {$\Omega_1$};
	\draw(3.5,2.5) node {$\Omega_2$};
	\draw(0,2.2) node {$\eta_1(\theta)$};
	\draw(6,2.2) node {$\eta_2(\theta)$};
  \end{tikzpicture}
  \caption{Decomposition through half-integer nodes}
\label{f-3}
\end{figure}

The second possibility, which is associated with decomposition through half-integer nodes, is depicted in Fig.~\ref{f-3}. 
In this case instead of (\ref{75}) we have  
\begin{equation}\label{76}
  R(\widetilde{\chi}_{\alpha}) = [0, 1/2, 1],
  \quad \alpha = 1,2,...,p .
\end{equation}
At node $\theta = \theta_i$ we use the difference approximation with the flux reduced by half.  
For the decomposition in the variable $x_1$  the operators of decomposition  (\ref{72}) seem like this  
\[
  A_1 y =
  \frac{1}{2 h_1^2} (y(x_1,x_2) - y(x_1-h_1,x_2)) 
\]
\[
  - \frac{1}{4h_2^2} (y(x_1,x_2+h_2) - 2y(x_1,x_2) - y(x_1,x_2-h_2)),
\]
\[
  A_2 y =
  - \frac{1}{h_1^2} (y(x_1+h_1,x_2) - y(x_1,x_2))
  + \frac{1}{2 h_1^2} (y(x_1,x_2) - y(x_1-h_1,x_2)) 
\]
\[
  - \frac{3}{4h_2^2} (y(x_1,x_2+h_2) - 2y(x_1,x_2) - y(x_1,x_2-h_2)),
  \quad x_1 = \theta_i .
\]
For  calculations in the domain  $\Omega_1$ (see Fig.~\ref{f-3})) we employ adjacent to the interface
data from the domain $\Omega_2$ --- at node $ \ theta = \ theta_i $.  
Thus, for this domain decomposition exchanges are minimal and coincide 
with the exchanges in the explicit scheme.  

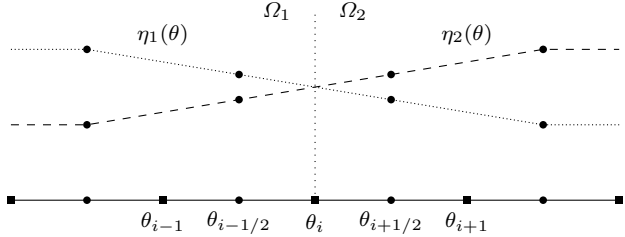
\begin{figure}[h]
  \begin{tikzpicture}
    \draw (-1,0) -- (7,0);
  	\foreach \x in {0,...,3}
	  \fill [black] (2*\x,0) circle (0.05);
  	\foreach \x in {0,...,1}
	  \fill [black] (6*\x,2) circle (0.05);
  	\foreach \x in {0,...,1}
	  \fill [black] (6*\x,1) circle (0.05);
  	\foreach \x in {0,...,1}
	  \fill [black] (2*\x+2,1+1./3) circle (0.05);
  	\foreach \x in {0,...,1}
	  \fill [black] (2*\x+2,1+2./3) circle (0.05);
  	\foreach \x in {0,...,4}
	  \fill [black] (2*\x-1-0.05,-0.05) rectangle +(0.1,0.1);
    \draw[densely dotted] (-1,2) -- (0,2) -- (2,1+2./3) -- (4,1+1./3) -- (6,1) -- (7,1);
    \draw[dashed] (-1,1) -- (0,1) -- (2,1+1./3) -- (4,1+2./3) -- (6,2) -- (7,2);
    \draw[dotted] (3,0) -- (3,2.5);
	\draw(3,-0.3) node {$\theta_i$};
	\draw(1,-0.3) node {$\theta_{i-1}$};
	\draw(5,-0.3) node {$\theta_{i+1}$};
	\draw(2,-0.3) node {$\theta_{i-1/2}$};
	\draw(4,-0.3) node {$\theta_{i+1/2}$};
	\draw(2.5,2.5) node {$\Omega_1$};
	\draw(3.5,2.5) node {$\Omega_2$};
	\draw(1,2.2) node {$\eta_1(\theta)$};
	\draw(5,2.2) node {$\eta_2(\theta)$};
  \end{tikzpicture}
  \caption{Decomposition through integer nodes with the width of overlapping  $3h$}
\label{f-4}
\end{figure}

The considered variants of decomposition  (\ref{75}), (\ref{76}) 
correspond to the minimum overlapping of subdomains.  
At the discrete level the width of overlapping is governed by the mesh size ($h$  and $2h$, respectively). 
Similar variants can be constructed for a higher overlapping of subdomains.  
For the decomposition presented in Fig.~\ref{f-4} we have  
\begin{equation}\label{77}
  R(\widetilde{\chi}_{\alpha}) = [0, 1/3, 2/3, 1],
  \quad \alpha = 1,2,...,p .
\end{equation}
Obviously, in this case we have a grater volume of data exchange, but at the same time the transition from one domain 
to another is much smoother.  
The latter allows us to expect a higher accuracy of the approximate solution.  

Consider now the results of approximate solving problem (\ref{65})--(\ref{67}), 
which has the exact solution  (\ref{74}). Let $n_1 = 2, \ n_2 = 1$, $ T = 0.01$
and the grid is square $N_1 = N_2$.
Calculations have been performed using regularized schemes with the additive (scheme (\ref{7}), (\ref{41}), (\ref{45}), (\ref{45}))   
and multiplicative (scheme  (\ref{7}), (\ref{41}), (\ref{48}), (\ref{49}))  perturbation of the  transition operator 
at $\sigma = 1$, as well as vector additive scheme (\ref{33}), (\ref{55}), (\ref{56}).
The results are compared with the difference solution obtained via implicit scheme 
(\ref{1}), (\ref{6}), (\ref{7})  at $\sigma = 1$.
The error of the approximate solution was estimated through value $\varepsilon(t^n) = \| y^n(\mathbf{x}) - u(\mathbf{x},t^n)\|$ 
at a particular time step.  

\begin{figure}[h]
  \includegraphics[width=0.8\textwidth,angle=-0]{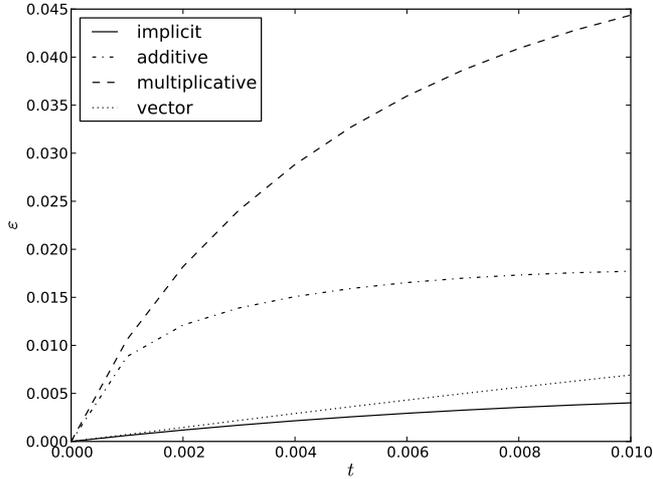}
  \caption{Error at $N_1 = N_2 = 32$ and $N = 10$}
\label{f-5}
\end{figure}

Considering decomposition  (\ref{75}) (the width of overlapping is $ h $) with 
the space grid $N_1 = N_2 = 32$  and time grid  $N = 10$ ($\tau = 0.001$), we can compare
the error norms of the difference solution obtained using different schemes (see Fig.~\ref{f-5}). 
Figure~\ref{f-6}--\ref{f-8} shows the local error at the final time moment.  
The error is localized in the area of overlapping and it is much lower for the vector  scheme of decomposition
in compare with  the additive and multiplicative variants of regularized additive schemes.  

\begin{figure}[h]
  \includegraphics[width=0.8\textwidth,angle=-0]{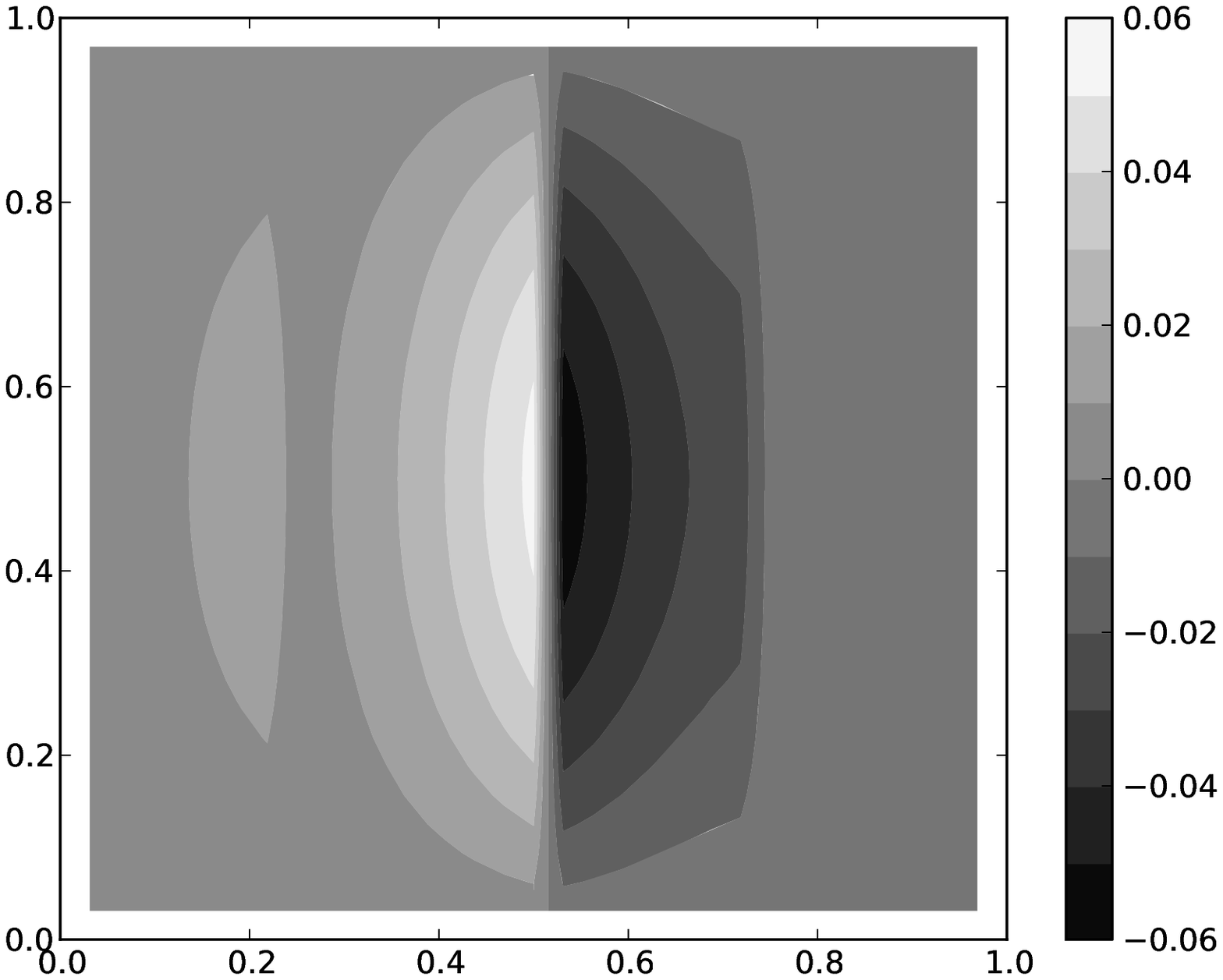}
  \caption{Error of scheme   (\ref{7}), (\ref{41}), (\ref{48}), (\ref{49})}
\label{f-6}
\end{figure}

\begin{figure}[h]
  \includegraphics[width=0.8\textwidth,angle=-0]{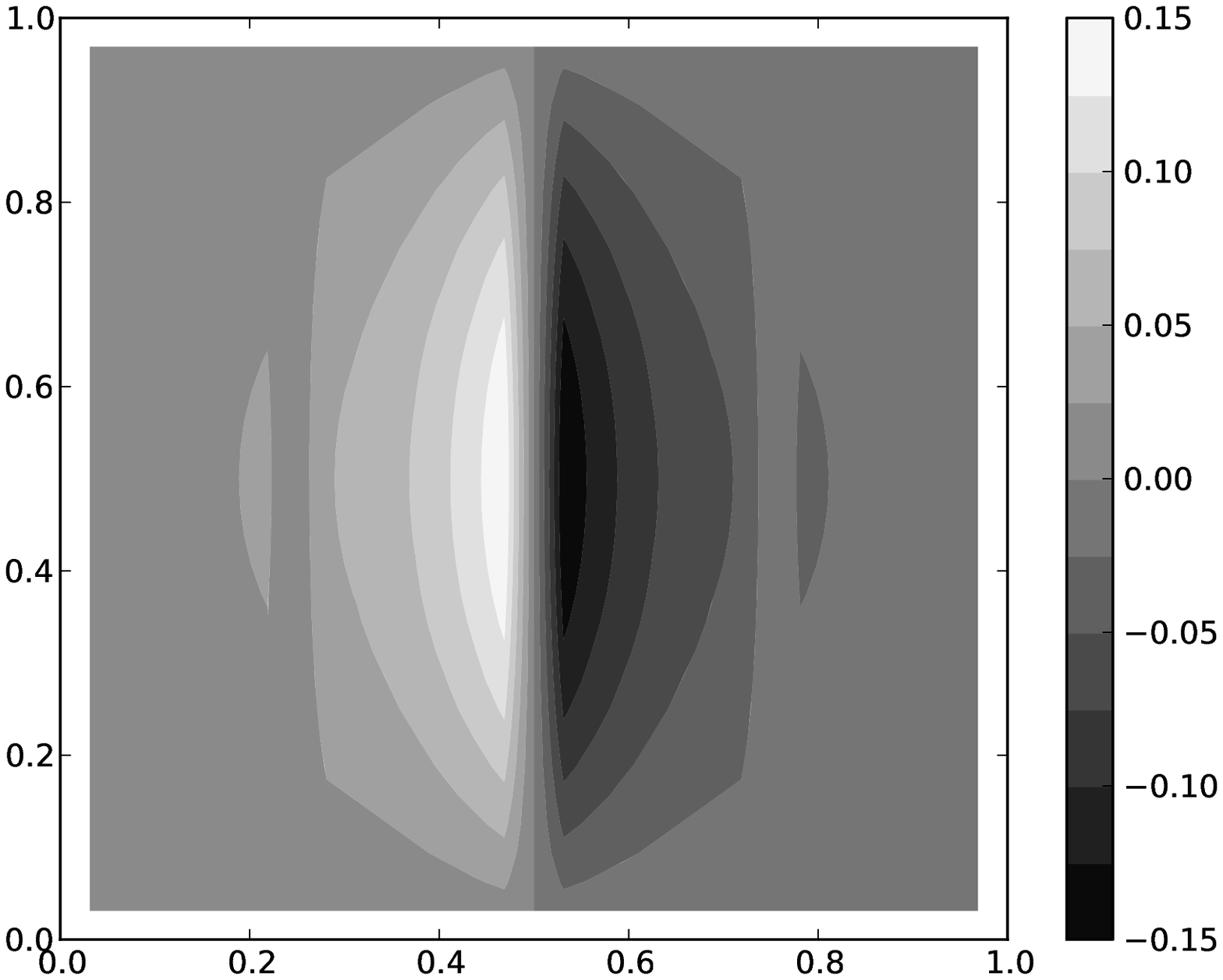}
  \caption{Error of scheme  (\ref{7}), (\ref{41}), (\ref{45}), (\ref{45})}
\label{f-7}
\end{figure}

\begin{figure}[h]
  \includegraphics[width=0.8\textwidth,angle=-0]{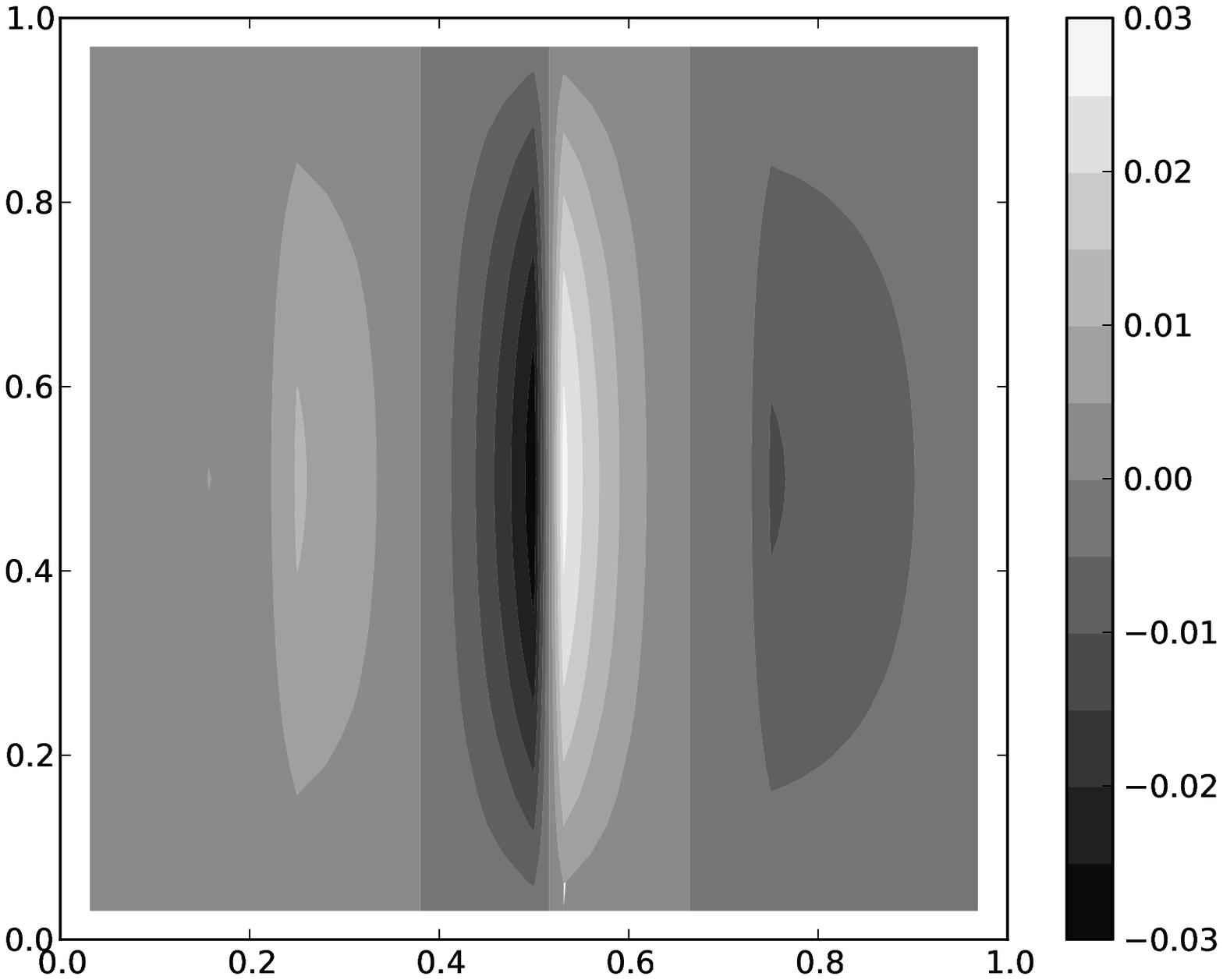}
  \caption{Error of scheme  (\ref{33}), (\ref{55}), (\ref{56})}
\label{f-8}
\end{figure}

\begin{figure}[h]
  \includegraphics[width=0.8\textwidth,angle=-0]{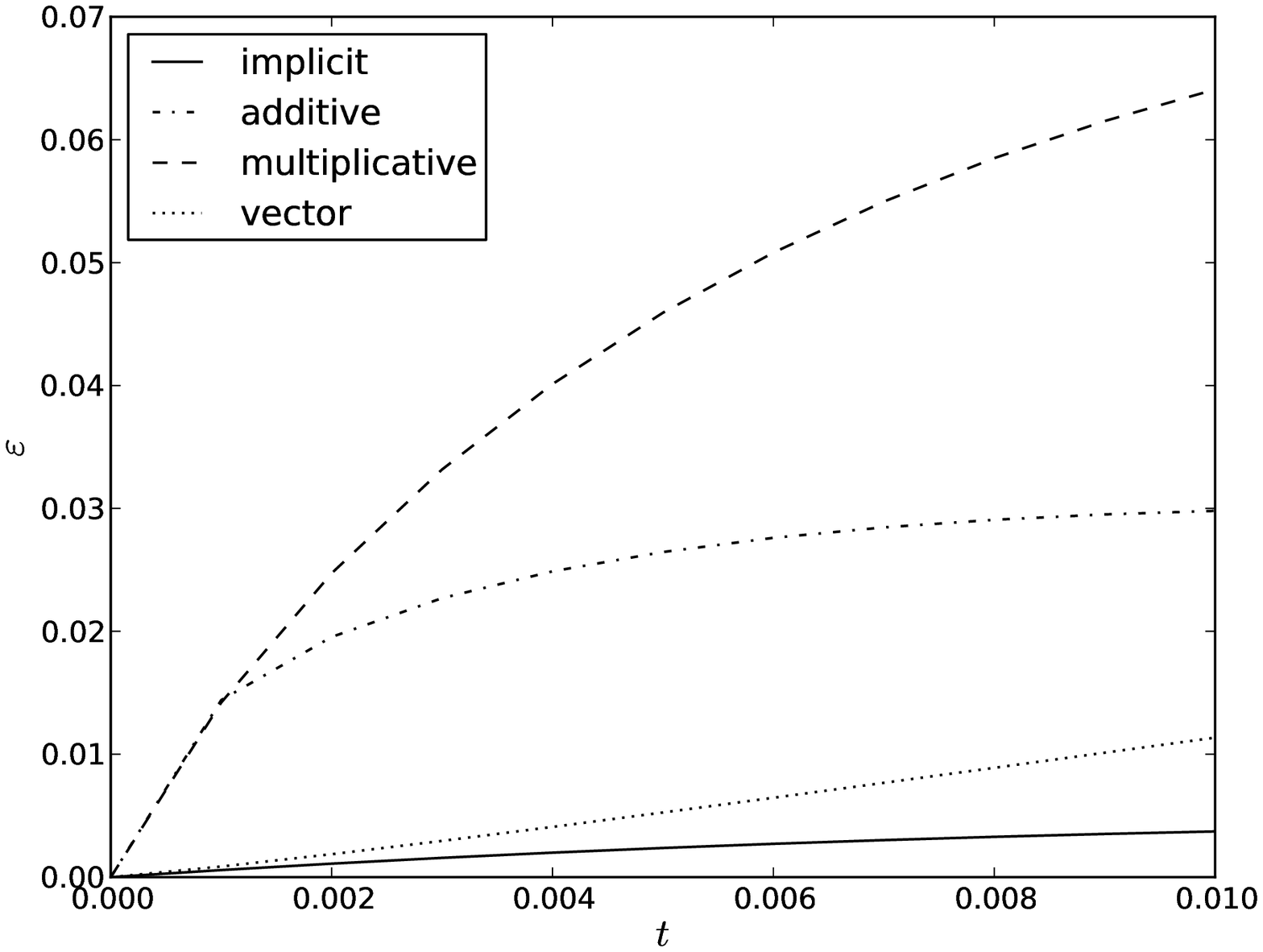}
  \caption{Error at $N_1 = N_2 = 64$ and $N = 10$}
\label{f-9}
\end{figure}

In contrast to the implicit scheme, the error of the approximate solution  grows with increasing the space grid
for domain decomposition schemes (Fig.~\ref{f-9}). 
In this case the width of overlapping is reduced by half.  

\begin{figure}[h]
  \includegraphics[width=0.8\textwidth,angle=-0]{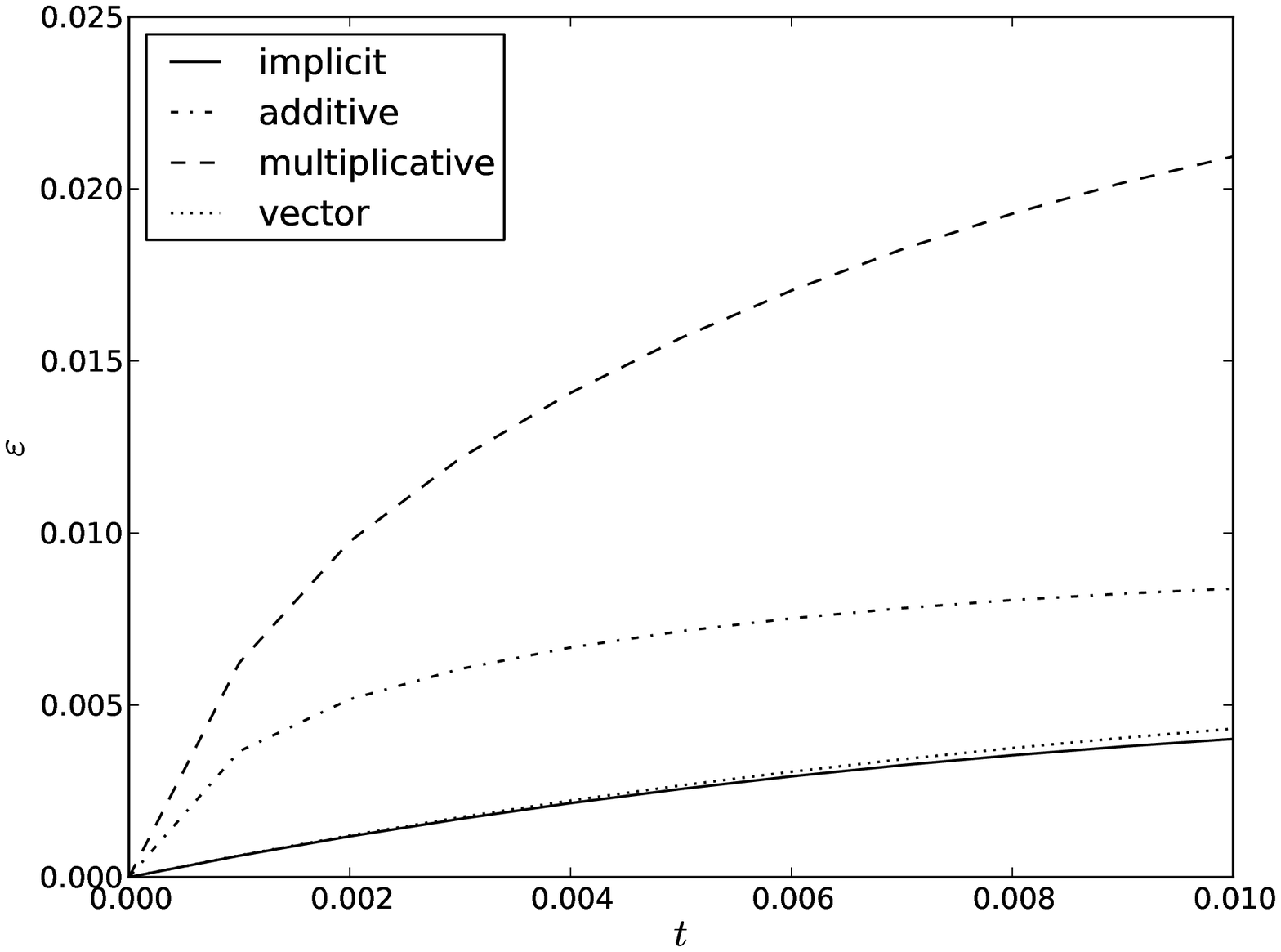}
  \caption{The error for $N_1 = N_2 = 32$, $N = 10$ and decomposition    $R = [0, 1/3, 2/3, 1]$}
\label{f-10}
\end{figure}

The dependence of results on the width of overlapping is shown in Fig.~\ref{f-10}.
It is easy to see that decomposition (\ref{77}) demonstrates the approximate solution
of essentially higher accuracy in compare with decomposition(\ref{75}) 
(compare  Fig.~\ref{f-5} with Fig.~\ref{f-10}). 

\section{Conclusions  }
\label{s-7}

\begin{enumerate}
  \item In this paper we have constructed the operators of domain decomposition for solving evolutionary problems.
Splitting of the general not self-adjoint nonnegative finite-dimensional  operator is performed 
separately for its self-adjoint and skew-symmetric parts.  
This preserves the property of nonnegativity for operator terms associated with individual subdomains.   

 \item There are constructed unconditionally stable regularized additive schemes for the Cauchy problem 
for evolutionary equations  of first order based on splitting of the problem operator into the sum 
of not self-adjoint nonnegative operators.
Regularization is based on the principle of regularization for operator-difference schemes with 
perturbation of the transition operator for the explicit scheme.  
Both additive and multiplicative splittings are considered. It was highlighted the relationship of such 
regularized schemes with the additive schemes of summarized  approximation: 
additively-averaged schemes as well as standard component-wise  splitting ones.  

 \item Vector additive schemes of full approximation are  selected among the splitting schemes for evolutionary equations.  
They are based on the transition to a system of similar problems
with a special component-wise organization for searching the 
approximate solution at the new time level.  

  \item The numerical solution of the boundary value 
problem for the parabolic equation in a rectangle was conducted.    
The calculations allow to compare various schemes of domain decomposition and to show
the accuracy dependence of the approximate solution on the width of overlapping.  
The vector additive scheme of domain decomposition demonstrates the best results in terms of accuracy.    
\end{enumerate}


\begin{thebibliography}{10}
\providecommand{\url}[1]{{#1}}
\providecommand{\urlprefix}{URL }
\expandafter\ifx\csname urlstyle\endcsname\relax
  \providecommand{\doi}[1]{DOI~\discretionary{}{}{}#1}\else
  \providecommand{\doi}{DOI~\discretionary{}{}{}\begingroup
  \urlstyle{rm}\Url}\fi

\bibitem{0712.65089}
Abrashin, V.: {A variant of the method of variable directions for the solution
  of multi- dimensional problems of mathematical-physics. I.}
\newblock Differ. Equations \textbf{26}(2), 243--250 (1990)

\bibitem{abrashin1998numerical}
Abrashin, V., Vabishchevich, P.: {Vector Additive Schemes for Second-Order
  Evolution Equations}.
\newblock Differential Equations \textbf{34}(12), 1673--1681 (1998)

\bibitem{Cai:1991:ASA}
Cai, X.C.: Additive {S}chwarz algorithms for parabolic convection-diffusion
  equations.
\newblock Numer. Math. \textbf{60}(1), 41--61 (1991)

\bibitem{Cai:1994:MSM}
Cai, X.C.: Multiplicative {S}chwarz methods for parabolic problems.
\newblock SIAM J. Sci Comput. \textbf{15}(3), 587--603 (1994)

\bibitem{Dryja:1991:SMP}
Dryja, M.: Substructuring methods for parabolic problems.
\newblock In: R.~Glowinski, Y.A. Kuznetsov, G.A. Meurant, J.~P{\'e}riaux,
  O.~Widlund (eds.) Fourth International Symposium on Domain Decomposition
  Methods for Partial Differential Equations. SIAM, Philadelphia, PA (1991)

\bibitem{0825.65066}
Kuznetsov, Y.: {New algorithms for approximate realization of implicit
  difference schemes}.
\newblock Sov. J. Numer. Anal. Math. Model. \textbf{3}(2), 99--114 (1988)

\bibitem{0766.65089}
Kuznetsov, Y.: {Overlapping domain decomposition methods for FE-problems with
  elliptic singular perturbed operators}.
\newblock {Fourth international symposium on domain decomposition methods for
  partial differential equations, Proc. Symp., Moscow/Russ. 1990, 223-241
  (1991)} (1991)

\bibitem{Laevsky}
Laevsky, Y.: Domain decomposition methods for the solution of two-dimensional
  parabolic equations.
\newblock In: Variational-difference methods in problems of numerical analysis,
  2, pp. 112--128. Comp. Cent. Sib. Branch, USSR Acad. Sci., Novosibirsk
  (1987).
\newblock In Russian

\bibitem{1152.15001}
Lax, P.D.: {Linear algebra and its applications. 2nd edition.}
\newblock {Pure and Applied Mathematics. A Wiley-Interscience Series of Texts,
  Monographs \&amp; Tracts. New York, NY: Wiley. xvi, 376~p. } (2007)

\bibitem{Marchuk:1990:SAD}
Marchuk, G.: Splitting and alternating direction methods.
\newblock In: P.G. Ciarlet, J.L. Lions (eds.) Handbook of Numerical Analysis,
  Vol. I, pp. 197--462. North-Holland (1990)

\bibitem{pre05281749}
Mathew, T.: {Domain decomposition methods for the numerical solution of partial
  differential equations}.
\newblock {Lecture Notes in Computational Science and Engineering 61. Berlin:
  Springer. xiii, 764~p.} (2008)

\bibitem{0931.65118}
Quarteroni, A., Valli, A.: {Domain decomposition methods for partial
  differential equations}.
\newblock {Numerical Mathematics and Scientific Computation. Oxford: Clarendon
  Press. xv, 360 p.} (1999)

\bibitem{0971.65076}
Samarskii, A.: {The theory of difference schemes}.
\newblock {Pure and Applied Mathematics, Marcel Dekker. 240. New York, NY:
  Marcel Dekker. 786 p.} (2001)

\bibitem{1018.65103}
Samarskii, A., Matus, P., Vabishchevich, P.: {Difference schemes with operator
  factors}.
\newblock {Mathematics and its Applications (Dordrecht). 546. Dordrecht: Kluwer
  Academic Publishers. x, 384 p.} (2002)

\bibitem{samarskii1989numerical}
Samarskii, A., Nikolaev, E.: {Numerical methods for grid equations}.
\newblock Birkh{\"a}user (1989)

\bibitem{0863.65056}
Samarskii, A., Vabishchevich, P.: {Vector additive schemes of domain
  decomposition for parabolic problems}.
\newblock Differ. Equations \textbf{31}(9), 1522--1528 (1995)

\bibitem{0928.65102}
Samarskii, A., Vabishchevich, P.: {Factorized finite-difference schemes for the
  domain decomposition in convection-diffusion problems}.
\newblock Differ. Equations \textbf{33}(7), 972--979 (1997)

\bibitem{samarskii1998regularized}
Samarskii, A., Vabishchevich, P.: {Regularized additive full approximation
  schemes}.
\newblock Doklady. Mathematics \textbf{57}(1), 83--86 (1998)

\bibitem{0963.65091}
Samarskii, A., Vabishchevich, P.: {Additive schemes for problems of
  mathematical physics (Additivnye skhemy dlya zadach matematicheskoj fiziki)}.
\newblock Moscow: Nauka. 320 p. (1999).
\newblock In Russian

\bibitem{vab_255}
Samarskii, A., Vabishchevich, P.: Domain decomposition methods for parabolic
  problems.
\newblock In: C.H. Lai, P.~Bjorstad, M.~Gross, O.~Widlund (eds.) Eleventh
  International Conference on Domain Decomposition Methods, pp. 341--347.
  DDM.org (1999)

\bibitem{SamVabConvection}
Samarskii, A., Vabishchevich, P.: {Numerical methods for solution of
  convection-diffusion problems (Chislennye metody resheniya zadach
  konvekcii-diffuzii)}.
\newblock Moscow: URSS. 247 p. (1999).
\newblock In Russian

\bibitem{samarskii1998stability}
Samarskii, A., Vabishchevich, P., Matus, P.: {Stability of Vector Additive
  Schemes}.
\newblock Doklady. Mathematics \textbf{58}(1), 133--135 (1998)

\bibitem{samarskii1992regularized}
Samarskii, A.A., Vabishchevich, P.N.: Regularized difference schemes for
  evolutionary second order equations.
\newblock Math. Models and Methods in Applied Sciences \textbf{2}(3), 295--315
  (1992)

\bibitem{0857.65126}
Smith, B.: {Domain decomposition. Parallel multilevel methods for elliptic
  partial differential equations}.
\newblock {Cambridge: Cambridge University Press. xii, 224 p.} (1996)

\bibitem{1069.65138}
Toselli, A., Widlund, O.: {Domain decomposition methods -- algorithms and
  theory}.
\newblock {Springer Series in Computational Mathematics 34. Berlin: Springer.
  xv, 450~p.} (2005)

\bibitem{0719.65072}
Vabishchevich, P.: {Difference schemes with domain decomposition for solving
  non-stationary problems}.
\newblock U.S.S.R. Comput. Math. Math. Phys. \textbf{29}(6), 155--160 (1989)

\bibitem{0723.65076}
Vabishchevich, P.: {Regional-additive difference schemes for nonstationary
  problems of mathematical physics}.
\newblock Mosc. Univ. Comput. Math. Cybern. (3), 69--72 (1989)

\bibitem{vab_138}
Vabishchevich, P.: Parallel domain decomposition algorithms for time-dependent
  problems of mathematical physics.
\newblock In: Advances in Numerical Methods and Applications, pp. 293--299.
  World Schientific (1994)

\bibitem{0838.65086}
Vabishchevich, P.: {Regionally additive difference schemes with a stabilizing
  correction for parabolic problems.}
\newblock Comput. Math. Math. Phys. \textbf{34}(12), 1573--1581 (1994)

\bibitem{0888.65097}
Vabishchevich, P.: {Finite-difference domain decomposition schemes for
  nonstationary convection-diffusion problems}.
\newblock Differ. Equations \textbf{32}(7), 929--933 (1996)

\bibitem{vabishchevich1996vector}
Vabishchevich, P.: {Vector additive difference schemes for first-order
  evolutionary equations}.
\newblock Computational mathematics and mathematical physics \textbf{36}(3),
  317--322 (1996)

\bibitem{1156.65084}
Vabishchevich, P.: {Domain decomposition methods with overlapping subdomains
  for the time-dependent problems of mathematical physics.}
\newblock Comput. Methods Appl. Math. \textbf{8}(4), 393--405 (2008)

\bibitem{0986.65510}
Vabishchevich, P., Verakhovskij, V.: {Difference schemes for component-wise
  splitting-decomposition of a domain}.
\newblock Mosc. Univ. Comput. Math. Cybern. \textbf{1994}(3), 7--11 (1994)

\bibitem{0209.47103}
Yanenko, N.: {The method of fractional steps. The solution of problems of
  mathematical physics in several variables}.
\newblock {Berlin-Heidelberg-New York: Springer Verlag, VIII, 160 p. with 15
  fig. } (1971)

\end{thebibliography}
%

\end{document}